\documentclass[12pt]{amsart}
\usepackage{amsfonts,amsmath,amsthm,amssymb}
\usepackage{latexsym}
\usepackage{graphics}
\newtheorem{theorem}{Theorem}[section]
\newtheorem{corollary}[theorem]{Corollary}

\newtheorem{lemma}[theorem]{Lemma}
\newtheorem{example}[theorem]{Example}

\newtheorem{proposition}[theorem]{Proposition}

\newtheorem{definition}[theorem]{Definition}
\theoremstyle{definition}

\theoremstyle{remark}
\newtheorem{rem}[theorem]{Remark}
\theoremstyle{remark}

\newcommand{\beql}[1]{\begin{equation}\label{#1}}
\newcommand{\eeq}{\end{equation}}

\begin{document}

\title[Exotic complex Hadamard matrices]{Exotic complex Hadamard matrices and their equivalence}

\author{Ferenc Sz\"oll\H{o}si}

\dedicatory{Dedicated to Professor Warwick de Launey on the occasion of his $50^{th}$ birthday.}

\date{August, 2009.}

\address{Ferenc Sz\"oll\H{o}si: Department of Mathematics and its Applications, Central European University, H-1051, N\'ador u. 9, Budapest, Hungary}\email{szoferi@gmail.com}

\thanks{This work was supported by Hungarian National Research Fund OTKA-K77748.\\
\indent Some results of this paper have appeared in the Master thesis of the author (in Hungarian) \cite{dipl}.}

\begin{abstract}
In this paper we use a design theoretical approach to construct new, previously unknown complex Hadamard matrices. Our methods generalize and extend the earlier results of \cite{HJ}, \cite{MW} and offer a theoretical explanation for the existence of some sporadic examples of complex Hadamard matrices in the existing literature. As it is increasingly difficult to distinguish inequivalent matrices from each other, we propose a new invariant, the fingerprint of complex Hadamard matrices. As a side result, we refute a conjecture of Koukouvinos et al.\ on $(n-8)\times (n-8)$ minors of real Hadamard matrices \cite{KLMS}.
\end{abstract}

\maketitle



\section{Introduction}
A complex Hadamard matrix $H$ is a square $n\times n$ matrix with arbitrary unimodular entries with complex orthogonal rows, i.e.\ $HH^\ast=nI$. They are the natural generalization of real Hadamard matrices. Constructions of complex Hadamard matrices are motivated by their various applications in quantum information theory \cite{wer}, harmonic analysis \cite{KM}, \cite{tao}, operator theory \cite{popa}, and combinatorics \cite{beth}. Let us recall that Hadamard matrices $H$ and $K$ are \emph{equivalent} if $H=P_1D_1KD_2P_2$ holds for some permutational matrices $P_1, P_2$ and unitary diagonal matrices $D_1, D_2$. An Hadamard matrix with its first row and first column consisting only of $1$s is said to be \emph{normalized}.

The pragmatic example for complex Hadamard matrices are the (rescaled) Fourier matrices $F_n$, and more generally the Butson type Hadamard matrices, all of whose entries are some fixed $m^{th}$ roots of unity \cite{Bu}. As the concept itself is still a discrete generalization of real Hadamard matrices, many aspects of the real theory can be successfully applied to them obtaining interesting structural-, existence- and non-existence type theorems \cite{dL}, \cite{W}. An infinite family of ``exotic'' (i.e.\ non Butson type) matrices, the circulant complex Hadamard matrices, were discovered in the early $90$s by the pioneering work of Bj\"orck \cite{BJ} and the follow-up papers \cite{HJ} and \cite{MW}. Recently it was pointed out by Di\c{t}\u{a} that the so-called generalized Kronecker product construction leads to matrices with free parameters, and therefore examples of non Butson type complex Hadamard matrices in composite dimensions \cite{dita}. The only examples of parametric families of prime orders are Petrescu's biunitaries \cite{petrescu}.

The outline of the paper is as follows. In Sections $2$--$4$ we use design theoretical methods to construct new examples of complex Hadamard matrices. Besides rediscovering some well-known matrices of small orders, we give a new and systematic proof for the existence of non-standard circulant complex Hadamard matrices of prime orders \cite{HJ}, \cite{MW}, and we present new, previously unknown families of complex Hadamard matrices as well. In particular, we show that the matrices $U_{15}, V_{15}$ and $W_{9A}$ are new to the literature. As it is increasingly difficult to distinguish inequivalent complex Hadamard matrices from each other, in Section $5$ we introduce a new invariant, the \emph{fingerprint} of complex Hadamard matrices. By combining Craigen's paper \cite{craigen} with the theory developed in Section $5$ we show, as a side result, that the $(n-8)\times(n-8)$ minors of real Hadamard matrices of order $n$ cannot take the values $k\cdot2^7\cdot n^{n/2-8}$ for $k\in\{28,29,30,31\}$ disproving a conjecture of Koukouvinos et al.\ \cite{KLMS}.

Through the paper we shall use the standard notations for well-known matrices such as $F_n, C_n, P_n$ etc. The description of these matrices is available in the online version of the Tadej--\.Zyczkowski catalogue of complex Hadamard matrices \cite{karol}, \cite{web}. Notations $U_{4m-1}$, $V_{4m-1}$, $W_{(4m+1)A}$ and $W_{(4m+1)B}$ are to be introduced in this paper.

\section{Block designs and complex Hadamard matrices}\label{sec2}
In this section we consider symmetric $2$-designs with the usual parameters $v,k,\lambda$. We always identify a block design $\mathcal{B}$ with its incidence matrix $B$. In this way we simply consider $0$-$1$ matrices of order $v$, with $k$ $1$s in every row and column satisfying $BB^\ast=(k-\lambda)I+\lambda J$, where $J$ is the all $1$ matrix. To exclude trivial cases, we suppose, as usual, that $\lambda<k<v$.

It is well-known that the algebraic object behind real Hadamard matrices is the $2$-$(4m-1,2m-1,m-1)$ Hadamard design. The combinatorial nature and the high level of internal symmetry of the design ensures that it can be trivially transformed to a real Hadamard matrix of order $4m$, by simply exchanging every $0$ with $-1$, and padding the obtained matrix with a full row and column of $1$s. It is natural to ask  whether we can construct complex Hadamard matrices in a similar fashion as well. In the following we give a characterization of complex Hadamard matrices, composed of two different entries. The main concept of this section relies on the following
\begin{definition}
We say that a block design $B$ \emph{induces} a complex Hadamard matrix if after exchanging every $0$ in $B$ with a fixed complex number $a$ of modulus one, the obtained matrix is complex Hadamard. We say that this is an \emph{induced} complex Hadamard matrix.
\end{definition}
Recall that a complex Hadamard matrix is \emph{regular}, if the absolute value of the sum of the entries in each row is constant.
The following two lemmata indicate that there is a one-to-one correspondence between block designs satisfying \eqref{1} and regular complex Hadamard matrices, composed of two different entries.
\begin{lemma}
A $2$-$(v,k,\lambda)$ design induces a complex Hadamard matrix if and only if
\begin{equation}\label{1}
\frac{v-\sqrt{v}}{2}\leq k\leq \frac{v+\sqrt{v}}{2}.
\end{equation}
\end{lemma}
\begin{proof}
Consider a block design $B$. After exchanging every $0$ with $a$, the orthogonality relation between any two rows of it reads
\begin{equation}\label{2}
\lambda+2(k-\lambda)\Re[a]+v+\lambda-2k=0.
\end{equation}
From this, by using $\lambda=k(k-1)/(v-1)$, one can express $\Re[a]$ in terms of $v,k$, and as $|a|=1$ should hold, we have
\begin{equation}\label{3}
\Re[a]=1-\frac{v(v-1)}{2k(v-k)}\geq -1.
\end{equation}
After rearranging the desired inequality follows.
\end{proof}
Clearly, induced Hadamard matrices are regular. We have the following
\begin{lemma}
Any regular complex Hadamard matrix $H$ of order $v$, composed of entries $\{1,a\}$, corresponds to a $2$-$(v,k,\lambda)$ block design.
\end{lemma}
\begin{proof}
If $a=-1$, then $H$ is real, and the corresponding object is the Menon design. Otherwise, if $a\neq -1$, then an easy (but tedious) analysis shows that the regularity condition implies that every rows and columns contains the same number of $1$s, say $k$. Now consider any two rows of $H$, and suppose they share a common $1$ in exactly $\lambda$ coordinates. The same orthogonality relation \eqref{2} implies that for these two rows we have $\lambda=k+v/(2\Re[a]-2)$. Therefore the value of $\lambda$ is independent of the choice of the rows.
\end{proof}
Observe that the parameters of the Hadamard design satisfy the inequality \eqref{1}. We have the following
\begin{theorem}\label{t1}
Suppose that we have a $2$-$(4m-1,2m-1,m-1)$ Hadamard design, represented by an incidence matrix $U$. Then, after replacing every $0$ with
\begin{equation}
a=-1+\frac{1}{2m}\pm\mathbf{i}\frac{\sqrt{4m-1}}{2m}
\end{equation}
in $U$ we obtain a complex Hadamard matrix $U_{4m-1}$.
\end{theorem}
It is well-known that the Paley-I construction gives rise to circulant Hadamard designs when $p\equiv 3\ (4)$ is a prime. Using this fact the authors of \cite{MW} constructed circulant complex Hadamard matrices of prime orders, without exploring the possibilities of constructing Hadamard matrices of composite orders as well. In particular, we have the following
\begin{corollary}[Munemasa--Watatani, \cite{MW}]
For every prime $p\equiv 3\ (4), p\geq 7$ there exists a circulant complex Hadamard matrix of order $p$, inequivalent to the Fourier matrix $F_p$.
\end{corollary}
\begin{rem}
It is an important open problem in the theory of operator algebras to decide whether there exist infinitely many inequivalent complex Hadamard matrices of prime orders \cite{popa}.
\end{rem}
Applying Theorem \ref{t1} for $m=1,2,3$ we get the matrices $F_3, C_{7A}, C_{7B}$ and $C_{11A}, C_{11B}$ respectively, while for $m=4$ we obtain from the $5$ inequivalent Hadamard designs at least $5$ new, previously unknown complex Hadamard matrices of order $15$. We exhibit a particular example obtained from the five fold tensor product of $F_2$:

\tiny
\begin{equation}
U_{15}=\left[
\begin{array}{ccccccccccccccc}
 a & 1 & a & 1 & a & 1 & a & 1 & a & 1 & a & 1 & a & 1 & a \\
 1 & a & a & 1 & 1 & a & a & 1 & 1 & a & a & 1 & 1 & a & a \\
 a & a & 1 & 1 & a & a & 1 & 1 & a & a & 1 & 1 & a & a & 1 \\
 1 & 1 & 1 & a & a & a & a & 1 & 1 & 1 & 1 & a & a & a & a \\
 a & 1 & a & a & 1 & a & 1 & 1 & a & 1 & a & a & 1 & a & 1 \\
 1 & a & a & a & a & 1 & 1 & 1 & 1 & a & a & a & a & 1 & 1 \\
 a & a & 1 & a & 1 & 1 & a & 1 & a & a & 1 & a & 1 & 1 & a \\
 1 & 1 & 1 & 1 & 1 & 1 & 1 & a & a & a & a & a & a & a & a \\
 a & 1 & a & 1 & a & 1 & a & a & 1 & a & 1 & a & 1 & a & 1 \\
 1 & a & a & 1 & 1 & a & a & a & a & 1 & 1 & a & a & 1 & 1 \\
 a & a & 1 & 1 & a & a & 1 & a & 1 & 1 & a & a & 1 & 1 & a \\
 1 & 1 & 1 & a & a & a & a & a & a & a & a & 1 & 1 & 1 & 1 \\
 a & 1 & a & a & 1 & a & 1 & a & 1 & a & 1 & 1 & a & 1 & a \\
 1 & a & a & a & a & 1 & 1 & a & a & 1 & 1 & 1 & 1 & a & a \\
 a & a & 1 & a & 1 & 1 & a & a & 1 & 1 & a & 1 & a & a & 1
\end{array}
\right],\ \ \ a=-\frac{7}{8}+\mathbf{i}\frac{\sqrt{15}}{8}.
\end{equation}
\normalsize
So far we have shown that regular complex Hadamard matrices, composed of two different entries, correspond to block designs. It is natural to ask whether there are examples of non-regular complex Hadamard matrices, composed of two different entries as well. Somewhat surprisingly, such matrices exist in the real case only.
\begin{lemma}
Suppose that we have a complex Hadamard matrix $H$, composed of two different entries, say $\{1,a\}$. If $a\neq -1$, then $H$ is regular.
\end{lemma}
\begin{proof}
Consider the inner product of any two rows of $H$:
\beql{}A+Ba+C\overline{a}+D=0,\eeq
where $A,B,C,D$ are integral numbers describing in how many coordinates meet the pairs $(1,1)$, $(1,a)$, $(a,1)$ and $(a,a)$, respectively, i.e.\ $A$ denotes the number of columns sharing a common $1$ in the rows considered, etc. As $a$ is non-real, we have $B=C$, and therefore $A+B=A+C$. This shows that the number of $1$s is the same in every two rows of $H$, hence $H$ is regular.
\end{proof}\begin{rem}
Let $n:=k-\lambda$ be the \emph{order} of a symmetric design. It is well known, that $4n-1\leq v$. On the other hand, equation \eqref{2} implies $v\leq 4n$. Therefore induced complex Hadamard matrices correspond to the Hadamard- and Menon designs only.\footnote{We thank Professor Chris Godsil for pointing out this fact.}
\end{rem}
\section{Conference matrices redux}\label{s3}
In the previous section we have described the rediscovery of the so called cyclic $p$-roots of ``index $2$'' type matrices of orders $p\equiv 3\ (4)$. It is quite natural then to try to look at the underlying structure of the circulant complex Hadamard matrices when $p\equiv 1\ (4)$ as well. To do this, we invoke an other design theoretical object, the conference matrices. Recall, that a conference matrix $C$ is a square matrix with $0$s on the main diagonal and $\pm1$ otherwise satisfying $CC^\ast=(n-1)I$. A conference matrix is normalized, if all the nonzero entries in the first row and column are $1$. We are interested in symmetric conference matrices of orders $4m+2$. The main result of this section is the following
\begin{theorem}\label{t2}
Given any normalized, symmetric conference matrix $C$ of order $4m+2\geq 6$ one can construct a complex Hadamard matrix in the following way: discard the first row and column of $C$ replace the $0$s with $1$, replace the off-diagonal $1$s with $c$ and replace the off-diagonal $-1$s with $\overline{c}$ where $c$ is an unimodular complex number with
\begin{equation}\label{6}
\Re[c]=-\frac{1}{4m}\pm\frac{\sqrt{4m+1}}{4m}.
\eeq
This procedure give rise to complex Hadamard matrices $W_{(4m+1)A}$ and $W_{(4m+1)B}$, depending on the sign of \eqref{6}.
\end{theorem}
\begin{proof}
Consider any normalized, symmetric conference matrix of order $4m+2$, and after neglecting its first row and column replace the diagonal $0$s to $1$, replace the off-diagonal $1$s and $-1$s to unimodular indeterminates $x$ and $y$ respectively, obtaining a matrix $W$. Our aim is to show that by setting $x=c, y=\overline{c}$ we get a complex Hadamard matrix. To do this, consider any two rows of $W$, and by pre- and post multiplying it by the same permutational matrix we can suppose that the considered rows are the first two. Now we have two essentially different cases, depending on either $W_{1,2}=x$ or $W_{1,2}=y$ hold. In the first case the scalar product of the rows reads
\begin{equation}\label{4}
2\Re[x]+2m\Re[x\overline{y}]+2m-1=0,
\end{equation}
while the second case leads to the equation
\begin{equation}\label{5}
2\Re[y]+2m\Re[x\overline{y}]+2m-1=0.
\end{equation}
Solving the system of equations \eqref{4}--\eqref{5} yields the desired result.
\end{proof}
\begin{rem}
The reader might amuse himself by checking that starting from $m=2$, matrices $W_{(4m+1)A}$ and $W_{(4m+1)B}$ are inequivalent.
\end{rem}
The construction corresponding to circulant conference matrices was discovered earlier in \cite{HJ}. The authors of that paper, however, did not explore the possibilities of constructing complex Hadamard matrices in composite dimensions. We have the following
\begin{corollary}[de la Harpe--Jones, \cite{HJ}]
For every prime $p\equiv 1\ (4), p\geq 13$ there exists a circulant complex Hadamard matrix of order $p$, inequivalent to the Fourier matrix $F_p$.
\end{corollary}
The existence of matrices $F_5, C_{13A}, C_{13B}$ follows from Theorem \ref{t1}. However, already for $m=2$ we obtain a new complex Hadamard matrix of order $9$. With choosing the positive sign in \eqref{6}, we have
\begin{equation}
W_{9A}=\left[
\begin{array}{ccccccccc}
 1 & c & c & c & c & \overline{c} & \overline{c} & \overline{c} & \overline{c} \\
 c & 1 & \overline{c} & \overline{c} & c & c & c & \overline{c} & \overline{c} \\
 c & \overline{c} & 1 & c & \overline{c} & c & \overline{c} & c & \overline{c} \\
 c & \overline{c} & c & 1 & \overline{c} & \overline{c} & c & \overline{c} & c \\
 c & c & \overline{c} & \overline{c} & 1 & \overline{c} & \overline{c} & c & c \\
 \overline{c} & c & c & \overline{c} & \overline{c} & 1 & c & c & \overline{c} \\
 \overline{c} & c & \overline{c} & c & \overline{c} & c & 1 & \overline{c} & c \\
 \overline{c} & \overline{c} & c & \overline{c} & c & c & \overline{c} & 1 & c \\
 \overline{c} & \overline{c} & \overline{c} & c & c & \overline{c} & c & c & 1\\
\end{array}
\right],\ \ \ \ c=\frac{1}{4}+\mathbf{i}\frac{\sqrt{15}}{4}.
\end{equation}
It is standard to show that this matrix is not included in the Tadej--\.Zyczkowski catalogue. In particular, it is inequivalent from $N_9$. The other choice of the sign in \eqref{6} would lead to a Butson type Hadamard matrix, composed of third roots of unity. 
\section{A generalization}
In Section \ref{sec2} starting from a single orthogonality condition, namely equation \eqref{2} we constructed complex Hadamard matrices with two different entries. In Section \ref{s3} we started from a less ``regular'' combinatorial object, and by introducing $3$ different entries we obtained Hadamard matrices with two essentially different orthogonality relations. In this section we start from block designs again, but we sacrifice some of the internal symmetries of the design in order to get still a feasible number of orthogonality equations. We have the following
\begin{theorem}\label{t3}
Given any normalized, symmetric, real Hadamard matrix $H$ of order $4m\geq 8$, one can construct a complex Hadamard matrix in the following way: discard the first row and column of $H$, replace all off-diagonal $-1$s with $b$ and replace all diagonal $1$s with $-b$, where $b$ is the following unimodular complex number:
\begin{equation}
b=-1+\frac{1}{2m-2}\pm\mathbf{i}\frac{\sqrt{4m-5}}{2m-2}.
\end{equation}
This procedure gives rise to complex Hadamard matrices $V_{4m-1}$.
\end{theorem}
\begin{proof}
The proof is similar to the calculations carried out during the proof of Theorem \ref{t2}, but slightly longer, as in this case we have $6$ essentially different orthogonality equations. After replacing the diagonal $1$s with $x$, the diagonal $-1$s with $y$ and finally, the off-diagonal $-1$s with $z$, the arising orthogonality equations to be satisfied (up to conjugation) are the following:
\begin{equation}\label{10}
2m\Re[z]+2\Re[x]+2m-3=0, \ \ \ (m\geq 3),
\end{equation}
\begin{equation}
2(m-1)\Re[z]+2\Re[x\overline{z}]+2m-1=0,
\end{equation}
\begin{equation}
2m\Re[z]+x+\overline{y}-\overline{z}+2m-2=0,
\end{equation}
\begin{equation}
2m\Re[z]+x\overline{z}+z\overline{y}-\overline{z}+2m-2=0,
\end{equation}
\begin{equation}
2(m-1)\Re[z]+2\Re[y]+2m-1=0,
\end{equation}
\begin{equation}\label{15}
2m\Re[z]+2\Re[y\overline{z}]+2m-3=0.
\end{equation}
Solving the system of equations \eqref{10}--\eqref{15} is straightforward. One can obtain either $\{x,y,z\}=\{1,a,a\}$ corresponding to the construction described in Theorem \ref{t1} or $\{x,y,z\}=\{-b,-1,b\}$, as desired.
\end{proof}
For $m=2,3$ the construction described in Theorem \ref{t3} leads to the rediscovery of Petrescu's matrix $P_7$ (see Example \ref{Ex}) and one of Nicoara's sporadic matrices $N_{11A}$. Also, for $m=4$ we can construct at least one new, previously unknown complex Hadamard matrix of order $15$, as follows
\tiny
\begin{equation}
V_{15}=\left[
\begin{array}{ccccccccccccccc}
 -1 & 1 & b & 1 & b & 1 & b & 1 & b & 1 & b & 1 & b & 1 & b \\
 1 & -1 & b & 1 & 1 & b & b & 1 & 1 & b & b & 1 & 1 & b & b \\
 b & b & -b & 1 & b & b & 1 & 1 & b & b & 1 & 1 & b & b & 1 \\
 1 & 1 & 1 & -1 & b & b & b & 1 & 1 & 1 & 1 & b & b & b & b \\
 b & 1 & b & b & -b & b & 1 & 1 & b & 1 & b & b & 1 & b & 1 \\
 1 & b & b & b & b & -b & 1 & 1 & 1 & b & b & b & b & 1 & 1 \\
 b & b & 1 & b & 1 & 1 & -1 & 1 & b & b & 1 & b & 1 & 1 & b \\
 1 & 1 & 1 & 1 & 1 & 1 & 1 & -1 & b & b & b & b & b & b & b \\
 b & 1 & b & 1 & b & 1 & b & b & -b & b & 1 & b & 1 & b & 1 \\
 1 & b & b & 1 & 1 & b & b & b & b & -b & 1 & b & b & 1 & 1 \\
 b & b & 1 & 1 & b & b & 1 & b & 1 & 1 & -1 & b & 1 & 1 & b \\
 1 & 1 & 1 & b & b & b & b & b & b & b & b & -b & 1 & 1 & 1 \\
 b & 1 & b & b & 1 & b & 1 & b & 1 & b & 1 & 1 & -1 & 1 & b \\
 1 & b & b & b & b & 1 & 1 & b & b & 1 & 1 & 1 & 1 & -1 & b \\
 b & b & 1 & b & 1 & 1 & b & b & 1 & 1 & b & 1 & b & b & -b
\end{array}
\right],\ \ \ b=-\frac{5}{6}+\mathbf{i}\frac{\sqrt{11}}{6}.
\end{equation}
\normalsize
Again, it is straightforward to see that this matrix is not included in the Tadej--\.Zyczkowski catalogue. These small order examples shows that Theorem \ref{t1} and Theorem \ref{t3} lead to inequivalent complex Hadamard matrices. We shall emphasize this fact in the following section.
\begin{rem}
One might try to investigate skew-symmetric matrices as well in a similar fashion as described in Theorem \ref{t3}. However, in that case nothing else than Theorem \ref{t1} can be obtained.
\end{rem}
\section{Equivalence of complex Hadamard matrices}
From time to time new complex Hadamard matrices appear in the literature, and it is increasingly difficult to realize that a ``newly discovered'' matrix is indeed inequivalent to all existing ones. The methods available to deal with the real case cannot be applied directly to complex matrices, as they either based on the combinatorial nature of real Hadamard matrices (and therefore cannot be generalized at all), or become computationally expensive after proper modifications. In this respect we seek for a new invariant which might help, in at least the small order cases, to quickly identify classes of inequivalent Hadamard matrices from each other. The basic method for detecting inequivalence of complex Hadamard matrices was the idea of Haagerup's \cite{haagerup}, who considered $2\times 2$ submatrices and a complex valued function defined on them. More precisely, the following set $\Lambda$ associated with a complex Hadamard matrix $[H]_{i,j}=h_{ij}$ of order $n$
\begin{equation}
\Lambda(H)=\left\{h_{ij}h_{kl}\overline{h}_{il}\overline{h}_{kj} : i,j,k,l=1,2,\hdots, n\right\}
\end{equation}
is invariant under equivalence. As an illustration, we show the following
\begin{proposition}
Complex Hadamard matrices $U_{4m-1}$ and $V_{4m-1}$ are inequivalent for $m\geq 2$.
\end{proposition}
\begin{proof}
Suppose that $4m-1$ is not a square. Normalize the matrix $U_{4m-1}$ and observe that there should be a matrix element $u\in\mathbb{Q}[\mathbf{i}\sqrt{4m-1}]\setminus\mathbb{Q[\mathbf{i}]}$. Clearly $u\in\Lambda(U_{4m-1})$ but $u\notin\Lambda(V_{4m-1})$. Otherwise, if $4m-1$ is a square, then $4m-5$ is not, and we can repeat, mutatis mutandis, the same proof starting with $V_{4m-1}$.
\end{proof}
On the one hand this invariant conveniently distinguish Butson Hadamard matrices composed from different roots of unity from each other, on the other hand, however, it cannot detect inequivalence between matrices composed from the same roots of unity. In particular, $\Lambda(H)=\left\{\pm1\right\}$ for every real Hadamard matrix (of order at least $2$). The features of complex Hadamard matrices characterized by the Haagerup set are ``local'' in the sense that only $2\times 2$ submatrices are considered. In order to capture at least some of the global properties of a complex Hadamard matrix, we shall consider higher order submatrices as well. The downside of our approach is that we do not introduce any fancy complex function defined on the submatrices, but we use the absolute value of the determinant instead. Let us introduce the following invariant of ours.
\begin{definition}\label{d2}
Let $H$ be a complex Hadamard matrix of order $n\geq4$. For fix $2\leq d\leq \left\lfloor n/2\right\rfloor$ let $I(d)$ be an index set and let us denote by $v_i(d)$ and $m_i(d), i\in I(d)$ the absolute values of the $d\times d$ minors of $H$, and their multiplicity, respectively. Then, the following ordered set
\begin{equation}
\Phi(H):=\{\{(v_i(d), m_i(d)) : i\in I(d)\} : d=2,\hdots, \left\lfloor n/2\right\rfloor\}
\end{equation}
is the \emph{fingerprint} associated to $H$. Note that the pair $(v_i(d),m_i(d))$ is ordered as well.
\end{definition}
Naturally, for every $d$ and for every $i,j\in I(d), i\neq j$ it follows that $v_i(d)\neq v_j(d)$. Note that $I(d)$ simply ``counts'' how many different values are taken by the absolute values of the $d\times d$ minors of $H$. 

As permuting and rephasing the rows and columns of a matrix does not affect the absolute values of its minors, it is straightforward to realize that $\Phi$ is a class function.
\begin{proposition}
The fingerprint of a complex Hadamard matrix is invariant under the usual equivalence.
\end{proposition}
Now we give an
\begin{example}[Petrescu's matrix, \cite{petrescu}]\label{Ex}
As it was explained earlier Theorem \ref{t3} implies the existence of the following matrix
\begin{equation}
P_7=\left[
\begin{array}{ccccccc}
 -1 & 1 & \omega & 1 & \omega & 1 & \omega \\
 1 & -1 & \omega & 1 & 1 & \omega & \omega \\
 \omega & \omega & -\omega & 1 & \omega & \omega & 1 \\
 1 & 1 & 1 & -1 & \omega & \omega & \omega \\
 \omega & 1 & \omega & \omega & -\omega & \omega & 1 \\
 1 & \omega & \omega & \omega & \omega & -\omega & 1 \\
 \omega & \omega & 1 & \omega & 1 & 1 & -1
\end{array}
\right], \omega=-\frac{1}{2}+\mathbf{i}\frac{\sqrt{3}}{2}.
\end{equation}
For $d=2$ it is easily seen that the absolute values of the $2\times 2$ minors take four different values only. Therefore the cardinality of $I(2)$ is $4$, and we are free to set $I(2)=\{1,2,3,4\}$. These four values with their respective multiplicities read $v_1=0, m_1=54; v_2=1, m_2=114; v_3=2, m_3=96$ and $v_4=\sqrt3, m_4=177$. For $d=3$ the absolute values of the $3\times 3$ minors take $12$ different values already. Hence, we can set the index set $I(3)=\{1,2,\hdots,12\}$. The fingerprint, associated to $P_7$ reads
\begin{equation}
\Phi(P_7)=\left\{\left\{(0, 54), (1, 114), (2, 96), (\sqrt{3}, 177)\right\}, \Big\{(0, 60), (1, 36), (2, 108), (3, 210),\right.
\end{equation}
\begin{equation*}
 \left.\left.(4, 110), (\sqrt{3}, 162), (2 \sqrt{3}, 
  216), (3 \sqrt{3}, 14), (\sqrt{7}, 111), (\sqrt{13}, 54), (\sqrt{19}, 
  36), (\sqrt{21}, 108)\right\}\right\}.
\end{equation*}
\end{example}
With the notations used in Definition \ref{d2} it is clear that the sets corresponding to the cases $d=1$ and $d=n$ are simply $\{(1,n^2)\}$ and $\{(n^{n/2},1)\}$ respectively, and therefore they are omitted from the definition of $\Phi$. Somewhat less obvious, however, why the remaining cases $\left\lfloor n/2\right\rfloor<d<n$ are excluded as well. This is explained in the following
\begin{proposition}\label{pr1}
For $1\leq d\leq n-1$ with the notations of Definition \ref{d2} we have
\begin{equation}\label{myF}
\{(v_i(n-d), m_i(n-d)) : i\in I(n-d)\}=\{(n^{n/2-d}v_i(d), m_i(d)) : i\in I(d)\}.
\end{equation}
\end{proposition}
The proof of Proposition \ref{pr1} easily follows from the following two linear algebraic lemmata. One suspects after realizing \eqref{myF} that in Hadamard matrices there is a one-to-one correspondence between the minors of size $d$ and $n-d$, which is indeed the case. The first lemma is the well-known generalized matrix determinant
\begin{lemma}\label{LS}
Suppose that $A$ is an invertible $n\times n$ matrix, $U,V$ are $n\times m$ matrices. Then
\begin{equation}
\mathrm{det}(A+UV^\ast)=\mathrm{det}(I+V^\ast A^{-1}U)\mathrm{det}(A).
\end{equation}
\end{lemma}
Despite our best efforts, we were unable to find any references to the following
\begin{lemma}\label{lx}
Given any unitary matrix
\begin{equation}
U=\left[\begin{array}{cc}
A & B\\
C & D\\
\end{array}\right]
\end{equation}
with blocks $A, B, C, D$ where $A$ and $D$ are not necessarily of the same size square matrices, then we have $|\mathrm{det}(A)|=|\mathrm{det}(D)|$.
\end{lemma}
\begin{proof}
As $U$ is unitary, we have
\begin{equation}\label{eqD1}
AA^\ast+BB^\ast=I
\end{equation}
\begin{equation}\label{eqD2}
CC^\ast+DD^\ast=I
\end{equation}
\begin{equation}\label{eqD3}
AC^\ast+BD^\ast=O
\end{equation}
\begin{equation}\label{eqD4}
A^\ast B+C^\ast D=O,
\end{equation}
where the last equation follows from the fact that $U^\ast$ is unitary as well. If both $A$ and $D$ are singular, then we are done. Otherwise we can suppose that, for example, $A$ is invertible. Hence from \eqref{eqD4} we have $B=-(A^\ast)^{-1}C^\ast D$, and then, by plugging into \eqref{eqD3} we get
\begin{equation}\label{eqD5}
A^\ast AC^\ast C=C^\ast DD^\ast C.
\end{equation}
Then, starting from equation \eqref{eqD1} we use repeatedly Lemma \ref{LS} to obtain
\begin{equation}
|\mathrm{det}(A)|^2=\mathrm{det}(I-BB^\ast)=\mathrm{det}(I-B^\ast B)=\mathrm{det}(I-D^\ast C(A^\ast A)^{-1}C^\ast D)=
\end{equation}
\[=\mathrm{det}(A^\ast A-C^\ast DD^\ast C)/\mathrm{det}(A^\ast A),\]
which, by formula \eqref{eqD5} equals to
\begin{equation}
\mathrm{det}(A^\ast A-A^\ast AC^\ast C)/\mathrm{det}(A^\ast A)=\mathrm{det}(I-C^\ast C)=\mathrm{det}(I-CC^\ast)=|\mathrm{det}(D)|^2.
\end{equation}
\end{proof}
Proposition \ref{pr1} has some striking applications in the determination of minors of real Hadamard matrices. For example, one might ask whether a $6\times 6$ $\pm1$ matrix with maximum determinant of $5\cdot 2^5$ can be embedded as a submatrix into a $8\times 8$ real Hadamard matrix $H$.\footnote{This problem and its elementary solution was presented by Jennifer Seberry in the International Conference on Design Theory and Applications in Galway, 2009.} This is easily seen to be impossible, as by formula \eqref{myF} $v_i(6)=8^{8/2-2}v_i(2), i\in I(2)$ should hold. As $v_i(2)\in\{0,2\}$ we conclude that $v_i(6)$ can assume the values $0$ or $4\cdot 2^5$ only. In particular, it is enough to study the distribution of minors up to size $n/2$ in Hadamard matrices. Apparently the authors of \cite{KLMS} were unaware of this fact, and although they observed that there is some connection between $j\times j$ and $(n-j)\times (n-j)$ minors (for small $j$), they conjectured that the $n-8$ minors can take the values $k\cdot 2^7\cdot n^{n/2-8}, k=1,2,\hdots, 32$. Again, by Proposition \ref{pr1} we have $v_i(n-8)=n^{n/2-8}v_i(8), i\in I(8)$, however, a result of Craigen \cite{craigen} shows that $v_i(8)/2^7\notin \{28,29,30,31\}$. In particular, $k\in \{28,29,30,31\}$ in the conjecture above is impossible.

We conclude our paper with the following remark: the classification of all cyclic $p$-roots of index $3$, and the related circulant complex Hadamard matrices have been completed very recently \cite{Ha2}. It would be interesting to see what kind of combinatorial object lies behind those matrices, if any.

\section*{Acknowledgements} The author is greatly indebted to the referees for their valuable comments which helped to improve this manuscript.

\end{document}